\documentclass[11pt]{amsart}
\usepackage{amsmath,amssymb,amsthm}
\usepackage{mathtools}
\usepackage{esint}
\usepackage{comment}
\allowdisplaybreaks

\usepackage{xcolor}
\usepackage{hyperref}

\theoremstyle{plain}
\newtheorem{theorem}{Theorem}[section]
\newtheorem{remark}[theorem]{Remark}

\newtheorem{corollary}[theorem]{Corollary}
\newtheorem{lemma}[theorem]{Lemma}
\newtheorem{proposition}[theorem]{Proposition}

\numberwithin{equation}{section}

\newcommand{\R}{\mathbb{R}}

\renewcommand{\S}{\mathbb{S}}
\newcommand{\RP}{\mathbb{RP}}

\newcommand{\ve}{\varepsilon}
\newcommand{\E}{\mathcal{E}}
\newcommand{\wh}[1]{\widehat{#1}}
\newcommand{\norm}[1]{\left\|#1\right\|}
\newcommand{\hess}{\mathrm{Hess}}
\DeclareMathOperator{\ric}{Ric}
\DeclareMathOperator{\vol}{Vol}
\DeclareMathOperator*{\argmax}{arg\,max}

\newcommand{\hd}{\widehat{\Delta}}
\newcommand{\hn}{\widehat{\nabla}}
\newcommand{\hg}{\widehat{g}}
\newcommand{\hw}{\widehat{W}}
\newcommand{\hm}{\widehat{M}}
\newcommand{\dv}{\,dV_{\widehat{g}}}
\newcommand{\da}{\,dA_{\widehat{g}}}

\newcommand{\dvol}{\,dV_g}
\newcommand{\darea}{\,dA_g}

\newcommand{\sls}{\{\rho\le r\}}
\newcommand{\ls}{\{\rho=r\}}
\newcommand{\bsls}{\{b\le r\}}
\newcommand{\bls}{\{b=r\}}

\title[Rigidity and volume pinching]{Rigidity and volume pinching for the sharp gradient estimate in positive Ricci curvature}

\author[Kim]{Junyoung Kim}
\address{Korea Advanced Institute of Science and Technology (KAIST)\\
Department of Mathematical Sciences\\
Daejeon, South Korea}
\email{junykim@kaist.ac.kr}

\author[Park]{Jiewon Park}
\address{Korea Advanced Institute of Science and Technology (KAIST)\\
Department of Mathematical Sciences\\
Daejeon, South Korea}
\email{jiewonpark@kaist.ac.kr}

\begin{document}

\begin{abstract}
Colding established a sharp gradient estimate for the Green function on manifolds with non-negative Ricci curvature, which was extended to positive Ricci curvature by Manea recently. We prove almost rigidity of such gradient estimates for positive Ricci curvature. We show that the average of the gradient deficit, namely $\fint (1-|\nabla b|^2)\,dV$, controls the volume deficit in an optimal quantitative manner; here $b$ is the distance-like function defined using the Green function. This implies a quantitative almost rigidity theorem for manifolds with $\ric\ge(n-1)g$ and a gap theorem for Einstein manifolds. We also prove a rigidity theorem for closed 4-dimensional Einstein manifolds by finding a new monotonicity formula.
\end{abstract}

\maketitle

\section{Introduction}

Gradient estimates for Green functions provide a useful bridge between analysis and geometry. On complete Riemannian manifolds with non-negative Ricci curvature, Colding introduced a distance-like function associated with the Green function and proved a sharp gradient estimate together with monotonicity formulas measuring the failure of the manifold to be conical \cite{Col12}. Recently, Manea obtained the positive-Ricci counterpart on closed manifolds \cite{Man25}.

In this paper we are also concerned with the positive-Ricci setting. Let $(M^n,g)$ be a closed connected Riemannian manifold of dimension $n\ge 3$ with $\ric\ge(n-1)g$. For $p\in M$, let $G_p$ be the positive Green function of the operator
\begin{equation*}
L=-\Delta+\frac{n(n-2)}{4},
\end{equation*}
with the normalization
\begin{equation}\label{eq:green-function-normalization}
LG_p=(n-2)\omega_{n-1}\delta_p,
\end{equation}
where $\omega_{n-1}=\vol(\S^{n-1})$. The function
\begin{equation}\label{eq:function b}
b_p=2\arcsin\left(\frac{1}{2}G_p^{\frac{1}{2-n}}\right),\quad b_p(p)=0,
\end{equation}
is well-defined, takes values in $[0,\pi]$, and satisfies the sharp estimate
\begin{equation}\label{eq:manea}
    |\nabla b_p|\le 1\quad\text{on $M\setminus\{p\}$}
\end{equation}
by \cite{Man25}. On the round sphere, $b_p$ is exactly the distance from $p$, so equality holds throughout the regular set. Conversely, equality at one point in \eqref{eq:manea} characterizes the round sphere \cite{Man25}.

We ask the natural question whether a suitable lower bound on $|\nabla b_p|$ would imply closeness to the round sphere. Obviously, one cannot impose a global positive lower bound on $|\nabla b_p|$, simply because the manifold is compact and $|\nabla b_p(x)| \to 1$ as $x \to p$ (cf. \cite[Lemma 2.4]{Man25}). Hence we consider the normalized integral gap
\begin{equation}\label{eq:integral gap}
\E_p(M,g):=\fint_M(1-|\nabla b_p|^2)\,dV_g,
\end{equation}
which, by \eqref{eq:manea}, always satisfies $0\le\E_p(M,g)\le 1$. The definition \eqref{eq:integral gap} permits the defect $(1-|\nabla b_p|^2)$ to concentrate on small regions. It is therefore not immediate that smallness of $\E_p(M,g)$ should control any global geometric quantity.

In this paper we prove two main theorems, obtained by two different methods. The first is a quantitative volume pinching theorem, valid in every dimension $n\ge 3$ under a positive Ricci lower bound. We show that the normalized integral gap controls the volume deficit linearly, and this linear dependence is optimal. As consequences of the first theorem, we deduce an almost rigidity theorem for manifolds with $\ric\ge(n-1)g$ and a gap theorem for Einstein manifolds. The second is a rigidity theorem for the Weyl tensor of closed Einstein 4-manifolds under a lower bound on the gradient, proved through a new monotonicity formula for the ratio of two Green functions.

Let us now state our first main result, Theorem \ref{thm2}, on the volume pinching. The result holds in any dimension $n\ge 3$ and requires only a positive Ricci lower bound, not necessarily Einstein.

\begin{theorem}[Volume pinching by the integral gap]\label{thm2}
For $n\ge 3$, there exists a constant $C_{n}>0$ with the following property. If $(M^n,g)$ is a closed connected Riemannian $n$-manifold with $\ric\ge(n-1)g$ and $p\in M$, then 
\begin{equation}\label{eq:thm2}
    1-\frac{\vol(M,g)}{\vol(\S^n)}\le C_{n}\E_p(M,g).
\end{equation}
\end{theorem}

The bound \eqref{eq:thm2} is optimal. Indeed, the round sphere of radius $R$ given by $(\S^n,R^2g_{\S^n})$ with $0<R\le 1$ satisfies $\E_p(\S^n,R^2g_{\S^n})\le C_n(1-R)$, whereas the left-hand side of \eqref{eq:thm2} equals $1-R^n=n(1-R)+O((1-R)^2)$, for $R$ close to 1; the details are in Subsection \ref{subsec: opt exp}. In order to prove Theorem \ref{thm2}, we first observe that the volume deficit is equivalent to a weighted version of the integral gap; see Proposition \ref{prop:vol deficit}. We then combine a maximum principle due to Wei--Ye \cite{WeiYe2007} and lower bounds on the Green function in order to bound this weighted integral gap by $\E_p(M,g)$.

Theorem \ref{thm2} has two corollaries.  In the setting of non-negative Ricci curvature, Honda--Peng \cite{HondaPeng} established almost rigidity of the gradient estimate. They showed that under the pointwise assumption $|\nabla b|\ge 1-\delta$, $M$ is measured Gromov--Hausdorff close to the Euclidean space; in fact their result holds more generally for non-parabolic $\mathsf{RCD}(0,N)$ spaces. Our first corollary is an analogous almost rigidity theorem in the positive-Ricci setting, in which only the smallness of the normalized integral gap \eqref{eq:integral gap} is assumed instead of a pointwise assumption.

\begin{corollary}[Almost rigidity for positive Ricci]\label{cor3}
For $n\ge 3$, there exist constants $\ve_n>0,C_n>0$ and $\alpha_n>0$ with the following property. Let $(M^n,g)$ be a closed connected Riemannian $n$-manifold with $\ric\ge(n-1)g$ and $p\in M$. If 
\begin{equation*}
\E_p(M,g)\le \ve_n,
\end{equation*}
then $M$ is diffeomorphic to $\S^n$ and
\begin{equation*}
d_{\mathrm{GH}}((M,g),(\S^n,g_{\S^n}))\le C_n\E_p(M,g)^{\alpha_n}.
\end{equation*}
\end{corollary}

Indeed, once we obtain the almost-maximal volume of Theorem \ref{thm2}, the quantitative volume pinching theorem of Aubry \cite[Theorem 4]{Aubry2005} yields both the diffeomorphism and a power-type Gromov--Hausdorff estimate in terms of the volume deficit. The diffeomorphism can also be obtained using Cheeger--Colding theory \cite{CheegerColding1997}. One may then take $\alpha_n=\beta(n)$, where $\beta(n)$ is the exponent of \cite{Aubry2005}.

For Einstein metrics, almost-maximal volume implies isometry to the round sphere. Combining Theorem \ref{thm2} with the sphere theorem for Einstein manifolds of Honda--Mondello \cite[Theorem C]{HondaMondello}, we obtain the following result; see also \cite[Theorem 3.12]{HondaMondino}, and \cite[Chapter 12]{Besse} for an alternative proof via the moduli space of Einstein metrics.

\begin{corollary}\label{thm3}
For $n\ge 3$, there exists a constant $\ve_n>0$ with the following property. Let $(M^n,g)$ be a closed connected Einstein $n$-manifold with $\ric=(n-1)g$ and $p\in M$. If
\begin{equation*}
\E_p(M,g)\le \ve_n,
\end{equation*}
then $(M,g)$ is isometric to the round sphere $(\S^n,g_{\S^n})$.
\end{corollary}

For our second main result, we consider $(M^4,g)$ a closed connected Einstein 4-manifold with $\ric=3g$. Fix distinct points $p,q\in M$, and consider the Green-function blow-up 
\begin{equation}\label{eq:blow-up-data}
\hg=G_q^2g\quad\text{on $\hm=M\setminus\{q\}$}\quad\text{and}\quad \rho=\left(\frac{G_q}{G_p}\right)^{1/2}.
\end{equation}
The metric $\hg$ is complete, scalar-flat, and asymptotically flat, while $\rho$ runs from $0$ at $p$ to infinity along the end corresponding to $q$. Thus $\rho$ plays the role of a radial function on the blow-up.

\begin{theorem}[Four-dimensional Weyl rigidity]\label{thm1}
There exists a universal constant $c>0$ with the following property. In the setting of \eqref{eq:blow-up-data}, suppose that, for some $\delta>0$, 
\begin{equation*}
|\hn\rho|_{\hg}\ge\delta\quad\text{on $M\setminus\{p,q\}$}\quad \text{and}\quad \sup_{\hm}\rho^2|\hw|_{\hg}\le c\delta^2.
\end{equation*}
Then $\hw\equiv 0$ on $\hm$, and $(M,g)$ is isometric to the round sphere $(\S^4,g_{\S^4})$.
\end{theorem}

The proof of Theorem \ref{thm1} adapts the ODE method of Lee--Park for complete Ricci-flat manifolds \cite{LP25} to a conformal blow-up of a positive Einstein manifold. The new tool is a weighted differential inequality for the density $G_q^{-1}|\hw|^2_{\hg}$, derived from Wu's conformally Einstein Weitzenb\"ock formula \cite{Wu2017}. Integrating this over sublevel sets of $\rho$ produces a function $F(r)$ satisfying an Euler-type differential inequality. The pole expansions at $p$ and at the asymptotically flat end impose incompatible boundary behavior unless $F$ vanishes identically. Consequently, the Weyl tensor vanishes, and the lower bound for $|\hn\rho|_{\hg}$ rules out the nontrivial spherical space form.

Finally, the hypotheses of Theorem \ref{thm1} are automatically satisfied when the normalized integral gap is sufficiently small, provided the second pole $q$ is chosen to be the maximum point of $b_p$.

\begin{theorem}\label{thm4}
Let $0<\delta<2$ and $\zeta>0$. Then there exists a constant $\ve_0=\ve_0(\delta,\zeta)>0$ with the following property. Let $(M^4,g)$ be a closed connected Einstein manifold with $\ric=3g$, fix $p\in M$, and choose 
\begin{equation*}
q\in\argmax_{x\in M}b_p(x).
\end{equation*}
If
\begin{equation*}
\E_p(M,g)\le\ve_0,
\end{equation*}
then
\begin{equation*}
|\hn\rho|_{\hg}\ge\delta\quad\text{on $M\setminus\{p,q\}$}\quad\text{and}\quad\sup_{\hm}\rho^2|\wh{W}|_{\hg}<\zeta.
\end{equation*}
\end{theorem}

Taking $\zeta<c\delta^2$ in Theorem \ref{thm4} and applying Theorem \ref{thm1} gives a second proof of Corollary \ref{thm3} in dimension four in the case $q\in\argmax_{x\in M}b_p(x)$, which does not rely on the sphere theorem of \cite{HondaMondello}.

\subsection*{Organization of the paper}
In Section \ref{section 2}, we establish the level set identities and prove Theorem \ref{thm2} with its corollaries. We also show that the exponent in Theorem \ref{thm2} is optimal. In Section \ref{section 3}, we carry out the derivation of the weighted Weitzenb\"ock inequality and the ODE argument in order to prove Theorem \ref{thm1}. Finally, Section \ref{section 4} is devoted to the proof of Theorem \ref{thm4}, which together with Theorem \ref{thm1} gives a second proof of Corollary \ref{thm3} without using the sphere theorem of Honda--Mondello.

\subsection*{AI disclosure} All mathematical contents of this paper, including the formulation of the problems, the statements of the results, and their proofs were developed by the authors, starting in late 2025. The manuscript was entirely written by the authors. In proving Lemma \ref{lem:volume lower bound} we needed a suitable Neumann maximum principle, and found out about \cite{WeiYe2007} during a search using ChatGPT 5.5 Pro. ChatGPT 5.6 Pro and ChatGPT 6 Astra were utilized for proofreading, including verifying the calculations done by the authors. We take full responsibility for the correctness of the manuscript.

\subsection*{Acknowledgment} This work was supported by the National Research Foundation of Korea(NRF) grant funded by the Korea government(MSIT) RS-2024-00346651. 

\section{Volume pinching by the integral gap}\label{section 2}
In this section we prove Theorem \ref{thm2} and deduce Corollaries \ref{cor3} and \ref{thm3}. Let $(M^n,g)$ be a closed connected Riemannian manifold of dimension $n\ge 3$ with $\ric\ge(n-1)g$. Fix $p\in M$, and let $G=G_p$ and $b=b_p$ be given by \eqref{eq:green-function-normalization}-\eqref{eq:function b}. Throughout this section, we write
\begin{equation*}
    s(r)=2\sin(r/2),\qquad c(r)=\cos(r/2),
\end{equation*}
so that the definition \eqref{eq:function b} reads
\begin{equation*}
G=s(b)^{2-n}.
\end{equation*}
Integrating the normalization \eqref{eq:green-function-normalization} over $M$ gives the identity
\begin{equation}\label{eq:integral of G}
\frac{n}{4}\int_M G\,dV_g=\omega_{n-1}.
\end{equation}

Set $m=\max_M b$. By \cite[Theorem 1.1]{Man25} we have $m\le \pi$, and if $m=\pi$, then $(M,g)$ is isometric to the round sphere, in which case Theorem \ref{thm2} and Corollaries \ref{cor3}, \ref{thm3} hold trivially. We therefore assume that $m<\pi$ for the remainder of this section.

For a regular value $r\in(0,m)$ of $b$, we define 
\begin{equation}\label{eq:JSH}
J(r)=\frac{n}{4}\int_{\bsls}G\dvol,\quad S(r)=\int_{\bls}|\nabla b|\darea,\quad H(r)=\int_{\bls}\frac{1-|\nabla b|^2}{|\nabla b|}\darea,
\end{equation}
and extend the domain of definition of $S$ and $H$ to all of $(0,m)$ by assigning arbitrary values at critical values of $b$. By the gradient estimate \eqref{eq:manea}, $S,H\ge0$. Also, the function $J$ is nondecreasing with $J(0^+)=0$ and, by \eqref{eq:integral of G}, $J(m)=\omega_{n-1}$.

\subsection{Level set identities}
We first record that the critical set of $b$ is negligible.

\begin{lemma}\label{lem:critical set of b}
We have
\begin{equation*}
\mathcal H^n\left(\{x\in M\setminus\{p\}\,:\,|\nabla b|(x)=0\}\right)=0.
\end{equation*}
\end{lemma}

\begin{proof}
On every relatively compact coordinate ball in $M\setminus\{p\}$, the function $G$ is a nonconstant solution of the homogeneous uniformly elliptic equation $LG=0$. The volume estimate for critical sets of linear elliptic equations then implies that $\{|\nabla G|=0\}$ has zero $n$-dimensional measure \cite{NaberValtorta2017}. Since $m<\pi$, the function relating $b$ to $G$, given by $x \mapsto s(x)^{2-n}$, has nonvanishing derivative on $M\setminus\{p\}$; hence $b$ and $G$ have the same critical set there.
\end{proof}

\begin{lemma}\label{lem:coarea}
For every non-negative Borel function $\varphi:(0,m)\rightarrow\R$,
\begin{equation}\label{eq:3.3}
\begin{aligned}
&\int_0^m \varphi(r)(S(r)+H(r))\,dr=\int_M\varphi(b)\,dV_g, \\
&\int_0^m\varphi(r) H(r)\,dr=\int_M\varphi(b)(1-|\nabla b|^2)\,dV_g.
\end{aligned}
\end{equation}
In particular, functions $J$ and $r\mapsto \vol(\{b\le r\})$ are absolutely continuous on $[0,m]$, and for a.e. $r\in(0,m)$,
\begin{equation}\label{eq:3.4}
J'(r)=\frac{n}{4}s(r)^{2-n}(S(r)+H(r)),\quad \frac{d}{dr}\vol(\bsls)=S(r)+H(r).
\end{equation}
\end{lemma}

\begin{proof}
The function $b$ is smooth on $M\setminus\{p\}$ with $0<b\le m$ there, and the level set $\{b=m\}$ consists of maximum points of $b$ and is hence contained in the critical set. Since the critical set has measure zero by Lemma \ref{lem:critical set of b}, the coarea formula gives, for every non-negative Borel function $\psi$ on $M\setminus\{p\}$, 
\begin{equation}\label{eq:lem3.2-1}
\int_M\psi\dvol=\int_{\{|\nabla b|>0\}}\psi \dvol=\int_0^m\int_{\bls}\frac{\psi}{|\nabla b|}\darea\,dr.
\end{equation}
By Sard's theorem, a.e. $r\in(0,m)$ is a regular value of $b$, and at every regular value,
\begin{equation*}
\int_{\bls}\frac{\varphi(b)}{|\nabla b|}\darea=\varphi(r)\int_{\bls}\frac{|\nabla b|^2+(1-|\nabla b|^2)}{|\nabla b|}\darea=\varphi(r)(S(r)+H(r)).
\end{equation*}
Thus the choices $\psi=\varphi(b)$ and $\psi=\varphi(b)(1-|\nabla b|^2)$ in \eqref{eq:lem3.2-1} give the first and second equations in \eqref{eq:3.3}, respectively.

Taking $\varphi=\frac{n}{4}s^{2-n}1_{(0,r]}$ and $\varphi=1_{(0,r]}$ in the first identity of \eqref{eq:3.3} and using $G=s(b)^{2-n}$, we obtain
\begin{equation*}
J(r)=\frac{n}{4}\int_0^r s^{2-n}(S+H)\,dt,\qquad \vol(\bsls)=\int_0^r(S+H)\,dt
\end{equation*}
for every $r\in(0,m)$. Since their values approach $\omega_{n-1}$ and $\vol(M)$ as $r \to m$, respectively, both integrands belong to $L^1(0,m)$, and the absolute continuity and \eqref{eq:3.4} therefore follow. 
\end{proof}

\begin{lemma}[cf. {\cite[Equation (4.6)]{Man25}}]\label{eq:lemma 3.3}
For every regular value $r\in(0,m)$ of $b$,
\begin{equation}\label{eq:3.8}
S(r)=s(r)^{n-1}c(r)^{-1}(\omega_{n-1}-J(r)).
\end{equation}
\end{lemma}

\begin{proof}
Since $G=s(b)^{2-n}$ and $s'=c$, we have $\nabla G=(2-n)s(b)^{1-n}c(b)\nabla b$. Hence, on a regular level set $\bls$ with unit normal $\nu=\nabla b/|\nabla b|$, we have
\begin{equation}\label{eq:lem3.3-1}
\int_{\bls}\partial_\nu G\darea=(2-n)s(r)^{1-n}c(r)S(r).
\end{equation}
Since $\nabla G\neq 0$ near $p$ from the local expansion of the Green function at the pole, every sufficiently small $\epsilon>0$ is a regular value of $b$, and 
\begin{equation}\label{eq:lem3.3-2}
\int_{\{b=\epsilon\}}\partial_\nu G\darea\rightarrow -(n-2)\omega_{n-1}\quad\text{as $\epsilon\rightarrow0$}
\end{equation}
by the normalization \eqref{eq:green-function-normalization}. Since $\Delta G=\frac{n(n-2)}{4}G$ on $\{\epsilon\le b\le r\}$, the divergence theorem implies
\begin{equation*}
\frac{n(n-2)}{4}\int_{\{\epsilon\le b\le r\}}G \dvol=\int_{\bls}\partial_\nu G\darea-\int_{\{b=\epsilon\}}\partial_\nu G\darea.
\end{equation*}
Letting $\epsilon\rightarrow0$ along regular values and using \eqref{eq:lem3.3-1}, \eqref{eq:lem3.3-2}, and $J(0^+)=0$, we obtain
\begin{equation*}
(n-2)J(r)=(2-n)s(r)^{1-n}c(r)S(r)+(n-2)\omega_{n-1},
\end{equation*}
which is \eqref{eq:3.8}.
\end{proof}

\subsection{A weighted integral gap and almost-maximal volume}
The key quantity in the proof of Theorem \ref{thm2} is the following weighted integral gap:
\begin{equation}\label{eq:weighted integral gap}
\eta:=\frac{n}{4}\int_M Gc(b)^2(1-|\nabla b|^2)\dvol.
\end{equation}

By \eqref{eq:integral of G} and Lemma \ref{lem:coarea}, we have
\begin{equation}
0\le \eta \le\omega_{n-1},\quad\text{and}\quad \eta=\frac{n}{4}\int_0^ms(t)^{2-n}c(t)^2H(t)\,dt.
\end{equation}
\begin{proposition}\label{prop:vol deficit}
For each $n\ge 3$, there exists a continuous positive function $K_n$ on $[0,\pi]$ such that
\begin{equation}\label{eq:vol deficit}
\vol(\S^n)-\vol(M)=\frac{n}{4}\int_0^mK_n(t)s(t)^{2-n}c(t)^2H(t)\,dt.
\end{equation}
In particular,
\begin{equation}\label{eq:vol deficit inequality}
\left(\min_{[0,\pi]}K_n\right)\eta\le \vol(\S^n)-\vol(M)\le\left(\max_{[0,\pi]}K_n\right)\eta.
\end{equation}
Moreover, $K_n(0)=\frac{\vol(\S^n)}{\omega_{n-1}}$, $K_n(\pi)=\frac{2^n(n-2)}{n(n+2)}$ and $K_4\equiv \frac{4}{3}$.
\end{proposition}
\begin{proof}
Extend $S$ and $H$ by zero on $[m,\pi]$, and extend $J$ by $J(r)=\omega_{n-1}$ there. Define 
\[D(r)=J(r)-\omega_{n-1}(1-c(r)^n),\quad 0\le r\le \pi.\]
Then $D$ is absolutely continuous, $D(0)=D(\pi)=0$, and $D(r)=\omega_{n-1}c(r)^n$ for $r\in[m,\pi]$. Substituting \eqref{eq:3.8} into \eqref{eq:3.4}, we find that, for a.e. $r\in(0,\pi)$,
\begin{equation*}
D'(r)+\frac{n}{2}\tan(r/2)D(r)=\frac{n}{4}s(r)^{2-n}H(r)
\end{equation*}
and
\begin{equation}\label{eq:D}
D(r)=\frac{n}{4}c(r)^n\int_0^rc(t)^{-n}s(t)^{2-n}H(t)\,dt\ge0,\quad 0\le r<\pi.
\end{equation}
Note that $0\le D(r)\le \omega_{n-1}c(r)^n$ since $J(r)\le \omega_{n-1}$ for $r\in[0,\pi]$.
In terms of $D$, \eqref{eq:3.8} becomes
\[S=\omega_{n-1}\sin^{n-1}r-s^{n-1}c^{-1}D\quad\text{a.e. on $(0,\pi)$}.\]
Integrating over $(0,\pi)$ and using $\vol(M)=\int_0^m(S+H)\,dr$, we obtain
\begin{equation}
\vol(\S^n)-\vol(M)=\int_0^\pi s(r)^{n-1}c(r)^{-1}D(r)\,dr-\int_0^mH(t)\,dt.
\end{equation}
By \eqref{eq:D} and Tonelli's theorem, the first integral equals
\[\frac{n}{4}\int_0^ms(t)^{2-n}c(t)^{-n}H(t)\left(\int_t^\pi \sin^{n-1}r\,dr\right)\,dt,\]
which is at most $\vol(\S^n)$ by $D\le \omega_{n-1}c^n$. Thus we obtain \eqref{eq:vol deficit} with 
\[K_n(t)=c(t)^{-n-2}\left(\int_t^\pi \sin^{n-1}r\,dr-\frac{4}{n}s(t)^{n-2}c(t)^n\right),\quad 0<t<\pi.\]
By the change of variable, we also obtain
\[K_n(t)=\frac{2^n(n-2)}{n}\int_0^1 y^{n+1}\left(1-c(t)^2y^2\right)^{(n-4)/2}\,dy.\]
The integrand in the last integral is bounded from above by $y^{n+1}$ for $n\ge 4$, and by $y^4(1-y^2)^{-1/2}$ for $n=3$, so $K_n$ is continuous and positive on $[0,\pi]$ by the dominated convergence theorem. The values $K_n(0), K_n(\pi)$ and the function $K_4$ can be calculated from these explicit formulas for $K_n$.
\end{proof}

In order to prove Theorem \ref{thm2}, we need the following intermediate lower bound for the volume.

\begin{lemma}\label{lem:volume lower bound}
For each $n\ge 3$, there exists a constant $C_n>0$ such that 
\[\eta\le C_n\E_p(M,g)^{1/n}.\]
In particular, there exists a constant $\delta_n>0$ such that $\vol(M)\ge\frac{1}{2}\vol(\S^n)$ whenever $\E_p(M,g)\le \delta_n$.
\end{lemma}

\begin{proof}
Let $h$ be the unique solution to
\begin{equation}\label{eq3.19}
Lh=nc(b)^2(1-|\nabla b|^2)\quad\text{on $M$}.
\end{equation}
Note that $h\in C^1(M)$, and $h\ge0$ if we test \eqref{eq3.19} with $\min\{h,0\}$. By the normalization \eqref{eq:green-function-normalization} and the symmetry of the Green function,
\begin{equation}\label{eq3.20}
h(p)=\frac{1}{(n-2)\omega_{n-1}}\int_M Gnc(b)^2(1-|\nabla b|^2)\dvol=\frac{4\eta}{(n-2)\omega_{n-1}}.
\end{equation}
Integrating \eqref{eq3.19} over $M$, we have
\begin{equation}\label{eq3.21}
\begin{aligned}
\fint_Mh\dvol&=\frac{4}{n(n-2)}\fint_M nc(b)^2(1-|\nabla b|^2)\dvol\\
&\le\frac{4}{n-2}\fint_M (1-|\nabla b|^2)\dvol=\frac{4}{n-2}\E_p(M,g).
\end{aligned}
\end{equation}
Since $h\ge0 $, \eqref{eq3.19} implies that
\begin{equation*}
\Delta h=\frac{n(n-2)}{4}h-nc(b)^2(1-|\nabla b|^2)\ge -nc(b)^2(1-|\nabla b|^2).
\end{equation*}
The Neumann-type maximum principle due to Wei--Ye \cite[Theorem C]{WeiYe2007} gives a constant $C_{n}>0$ such that
\begin{equation*}
\sup_Mh\le\fint_Mh\dvol+C_{n}\left(\fint_M\left(nc(b)^2(1-|\nabla b|^2)\right)^n\dvol\right)^{1/n}.
\end{equation*}
Since $0\le c(b)^2(1-|\nabla b|^2)\le 1$, the integrand of the last term is bounded by $n^n(1-|\nabla b|^2)$. Consequently the second term is bounded by $C_{n}n\E_p(M,g)^{1/n}$. Combining this with $\E_p(M,g)\le \E_p(M,g)^{1/n}$, we obtain
\begin{equation}
h(p)\le C_n\E_p(M,g)^{1/n}.
\end{equation}
This proves the first assertion. The second follows from \eqref{eq:vol deficit inequality}.
\end{proof}

We next establish a pointwise lower bound for the integrand $Gc(b)^2(1-|\nabla b|^2)$, which follows from the Bochner formula as in \cite[Equation (1.7)]{Man25}.

\begin{lemma}\label{lem2.6}
Let $u=Gc(b)^2(1-|\nabla b|^2)$. Then, on $M\setminus\{p\}$,
\[u\ge\frac{2^{2-n}}{\omega_{n-1}}\eta.\]
\end{lemma}

\begin{proof}
Let $v=s(b)=G^{1/(2-n)}$ so that $u=G(1-|\nabla v|^2-v^2/4)$. By \cite[Equation (1.7)]{Man25}, we have $Lu=\mathcal{F}$ on $M\setminus\{p\}$, where
\[\mathcal{F}:=\frac{1}{2}v^{-n}\left(\left|\hess_{v^2}-\frac{\Delta v^2}{n}g\right|^2+(\ric-(n-1)g)(\nabla v^2,\nabla v^2)\right)\ge0.\]
Since $v\rightarrow0$ and $|\nabla v|\rightarrow 1$ at $p$, we have $0\le u\le G$ and $u=o(r^{2-n})$, where $r=d(p,\cdot)$. 

Let $\chi_\ve$ be a cutoff function vanishing on $B_\ve(p)$, equal to one outside $B_{2\ve}(p)$, and satisfying $|\nabla\chi_\ve|\le C\ve^{-1}$ and $|\Delta\chi_\ve|\le C\ve^{-2}$. For $\varphi\in C^{\infty}(M)$ a nonnegative function, if we test $Lu=\mathcal{F}$ against $\chi_\ve\varphi$, then
\begin{align*}\int_M\chi_\ve\varphi\mathcal{F}\dvol&=\int_M uL(\chi_\ve\varphi)\dvol\\&=\int_M\chi_\ve u L\varphi\dvol -\int_M\left(u\varphi\Delta\chi_\ve+2u\langle\nabla\chi_\ve,\nabla\varphi\rangle\right)\dvol.\end{align*}
The first term goes to $\int_M uL\varphi\dvol$ as $\ve\rightarrow0^+$. By construction of $\chi_\ve$ and using $u= o(r^{2-n})$, the second term is bounded by
\begin{align*}
\int_M\left(|u\varphi\Delta\chi_\ve|+|2u\langle\nabla\chi_\ve,\nabla\varphi\rangle|\right)\dvol\le C\ve^{-2}\int_{B_{2\ve}(p)\setminus B_\ve(p)}u\dvol \le o(1).
\end{align*}
Hence, we have 
\begin{equation*}
\int_M\varphi\mathcal{F}\dvol=\int_M uL\varphi\dvol,\quad \text{for every $\varphi\in C^{\infty}(M), \varphi\ge 0$.}
\end{equation*}
Taking $\varphi\equiv 1$ gives 
\[\int_M\mathcal{F}\dvol=\frac{n(n-2)}{4}\int_Mu\dvol=(n-2)\eta.\]
Finally, for $x\in M\setminus\{p\}$, if we let $G_x$ be the Green function with pole at $x$ normalized as in \eqref{eq:green-function-normalization}, then $G_x\ge 2^{2-n}$ by \cite[Proposition 2.2]{Man25}, and the Green representation gives
\begin{align*}
u(x)&=\frac{1}{(n-2)\omega_{n-1}}\int_M G_x\,\mathcal{F}\dvol\\
&\ge \frac{2^{2-n}}{(n-2)\omega_{n-1}}\int_M\mathcal{F}\dvol=\frac{2^{2-n}}{\omega_{n-1}}\eta,
\end{align*}
as desired.
\end{proof}

\begin{proof}[Proof of Theorem \ref{thm2}]
As noted at the beginning of this section, we may assume $m<\pi$. Since $c(b)\le 1$, Lemma \ref{lem2.6}, the Cauchy--Schwarz inequality, and \eqref{eq:integral of G} give
\begin{align*}
\vol(M)\E_p(M,g)&=\int_M\frac{u}{Gc(b)^2}\dvol\ge\frac{2^{2-n}\eta}{\omega_{n-1}}\int_M \frac{1}{G}\dvol\\
&\ge\frac{2^{2-n}}{\omega_{n-1}}\frac{\vol(M)^2}{\int_M G\dvol}=\frac{n2^{-n}}{\omega_{n-1}^2}\eta \vol(M)^2.
\end{align*}
Let $\delta_n$ be the constant in Lemma \ref{lem:volume lower bound} so that if $\E_p(M,g)\le \delta_n$, then $\vol(M)\ge\frac{1}{2}\vol(\S^n)$. Combining with the above, we have that $\eta\le C_n\E_p(M,g)$ for $C_n>0$. Then \eqref{eq:vol deficit inequality} gives
\[1-\frac{\vol(M)}{\vol(\S^n)}\le C_n\E_p(M,g),\]
up to changing $C_n>0$. If $\E_p(M,g)>\delta_n$, then it is obvious that
\[1-\frac{\vol(M)}{\vol(\S^n)}\le 1<\delta_n^{-1}\E_p(M,g).\]
This completes the proof.
\end{proof}

We now prove the two corollaries of Theorem \ref{thm2}.

\begin{proof}[Proof of Corollary \ref{cor3}]
By Theorem \ref{thm2}, 
\begin{equation*}
\vol(\S^n)-\vol(M,g)\le C_{n}\vol(\S^n)\E_p(M,g).
\end{equation*}
Hence if $\ve_n$ is sufficiently small, then the volume of $(M,g)$ is almost maximal. Aubry's volume pinching theorem \cite{Aubry2005} shows that $\pi_1(M)$ is trivial, that $M$ is diffeomorphic to $\S^n$, and that
\begin{equation*}
d_{\mathrm{GH}}((M,g),(\S^n,g_{\S^n}))\le C(n)(\vol(\S^n)-\vol(M,g))^{\beta(n)},
\end{equation*}
where $C(n)>0$ and $\beta(n)>0$ are the constants of \cite{Aubry2005}.
\end{proof}

\begin{proof}[Proof of Corollary \ref{thm3}]
By Theorem \ref{thm2}, if $\E_p(M,g)\le \ve_n$ with $\ve_n$ sufficiently small, then $\vol(M,g)$ is arbitrarily close to $\vol(\S^n)$. Since we assumed that $(M,g)$ is Einstein with $\ric=(n-1)g$, the sphere theorem of Honda--Mondello \cite[Theorem C]{HondaMondello} implies that $(M,g)$ is isometric to the round sphere $(\S^n,g_{\S^n})$. 
\end{proof}

\begin{remark}\label{rmk3.5}
When $n=4$, we have $K_4\equiv \frac{4}{3}$ and so
\begin{equation}\label{eq:rmk2.7}\eta=\omega_{3}\left(1-\frac{\vol(M,g)}{\vol(\S^4)}\right),\end{equation}
which also follows from \cite[Proposition 6.1]{Man25} by taking $n=4$. Thus Proposition \ref{prop:vol deficit} extends the four-dimensional identity \eqref{eq:rmk2.7} to every dimension $n\ge 3$.

The choice of the weight $Gc(b)^2$ in \eqref{eq:weighted integral gap} also agrees with the formulation of Gursky--Malchiodi \cite{GurskyMalchiodi}. When $n=4$, let $\wh{g}=G^2g$ be the conformal blow-up on $M\setminus\{p\}$ (with pole at $p$) and let $F=G^{-1}(1+|\hn G|_{\hg}^2)-4$ be the function introduced in \cite[Equation (3.1)]{GurskyMalchiodi}. Substituting $G=s(b)^{-2}$ gives the pointwise identity $-F=4c(b)^2(1-|\nabla b|^2)$. Hence the estimate \eqref{eq:manea} is equivalent to the estimate $|\hn G|_{\hg}^2\le 4G-1$ of \cite[Proposition 3.3]{GurskyMalchiodi} in dimension four, and the weighted integral gap $\eta$ becomes 
\begin{equation*}
\eta=-\frac{1}{4}\int_M FG\dvol.
\end{equation*}
So the weight $Gc(b)^2$ in \eqref{eq:weighted integral gap} is natural in light of \cite{GurskyMalchiodi}. Moreover, for Einstein metrics, we can combine the mass formula \cite[Theorem 1.1]{GurskyMalchiodi} with the identity above to obtain that $\eta\le 2\pi^2 m(\hg)$, where $m(\hg)$ denotes the ADM mass of the blow-up.
\end{remark}

\subsection{Optimality of the exponent} \label{subsec: opt exp}
In this subsection we show that the linear bound in Theorem \ref{thm2} is optimal. We fix $p\in \S^n$ and denote $g_R=R^2g_{\S^n}$ for $0<R\le 1$. Since $\ric_{g_R}=(n-1)R^{-2}g_R\ge (n-1)g_R$, the rescaled sphere $(\S^n,g_R)$ belongs to the class under consideration of Theorem \ref{thm2}. We claim that there is a constant $C_n>0$ such that $\E_p(\S^n,g_R)\le C_n(1-R)$ for $R$ sufficiently close to $1$.

Let $a=\frac{n(n-2)}{4}$. For $\lambda>0$, let $H_\lambda$ be the Green function on $(\S^n,g_1)$ satisfying $(-\Delta_{g_1}+\lambda)H_\lambda=(n-2)\omega_{n-1}\delta_p$. We note that 
\begin{equation*}
    H_a(\theta)=(2\sin(d_{g_1}(p,\cdot)/2))^{2-n}.
\end{equation*} A direct scaling gives $G_{p,R}=R^{2-n}H_{aR^2}$. If we set $v_\lambda=H_\lambda^{1/(2-n)}$, then $v_a=s$ and $s(b_{p,R})=Rv_{aR^2}$. Moreover, $H_\lambda\ge H_a$ for $\lambda\le a$, and thus $v_\lambda\le v_a=s$.

By rotational symmetry, $H_\lambda$ is a function of the distance alone. Thus we may rewrite the equation $(-\Delta_{g_1}+\lambda)H_\lambda=(n-2)\omega_{n-1}\delta_p$ as
\[(\sin^{n-1}(\theta)H_\lambda'(\theta))'=\lambda H_\lambda(\theta)\sin^{n-1}(\theta),\]
which integrates to
\[-H_\lambda'(\theta)=\lambda\sin^{1-n}\theta\int_\theta^\pi H_\lambda(t)\sin^{n-1}t\,dt.\]
With some elementary algebra, this implies that for $\lambda=aR^2$ and $R$ close to 1,
\[|v_\lambda-s|\le C_n(a-\lambda)s,\quad |v_\lambda'-c|\le C_n(a-\lambda)c.\]
Differentiating $s(b_{p,R})=Rv_{aR^2}$, we have
\[1-|\nabla^{g_R}b_{p,R}|^2=\frac{1-(v_\lambda')^2-R^2v_\lambda^2/4}{1-R^2v_\lambda^2/4}.\]
Combining with the above and noting that $a-\lambda=a(1-R^2)$, we have
\[0\le 1-(v_\lambda')^2-R^2v_\lambda^2/4\le C_n(1-R^2).\]
Since
\[1-\frac{R^2v_\lambda^2}{4} \ge 1-\frac{R^2s^2}{4}\ge 1-\frac{s^2}{4}=c(\theta)^2,\]
we arrive at
\[0\le 1-|\nabla^{g_R}b_{p,R}|^2\le C_n\frac{1-R^2}{c(\theta)^2}.\]
The function $c(\theta)^{-2}$ is integrable over $\S^n$ for every $n\ge 3$. Integrating and dividing by $\vol(\S^n,g_{R})$ we obtain
\[\E_p(\S^n,g_R)\le C_n(1-R^2)\le 2C_n(1-R).\]

On the other hand, we have
\[1-\frac{\vol(\S^n,g_R)}{\vol(\S^n,g_1)}=1-R^n=n(1-R)+O((1-R)^2).\]
If a uniform estimate with some $\gamma>1$ held, then 
\[1-R^n\le C_{n,\gamma}\E_p(\S^n,g_R)^{\gamma}\le C(1-R)^\gamma,\]
which is impossible as $R\rightarrow 1$. Thus the exponent 1 in Theorem \ref{thm2} is optimal.

\section{Weyl rigidity for Einstein 4-manifolds}\label{section 3}
Throughout this section, $(M^4,g)$ is a closed connected Einstein 4-manifold with $\ric_g=3g$. For $a\in M$, we denote by $G_a$ the positive Green function for $L=-\Delta+2$ with pole at $a$ and normalization $LG_a=2\omega_3\delta_a$. Fix two distinct points $p,q\in M$ and set
\begin{equation*}
\hg=G_q^2g\quad\text{on $\hm:=M\setminus\{q\}$},\quad \rho=\left(\frac{G_q}{G_p}\right)^{1/2}\quad\text{on $M\setminus\{p,q\}$.}
\end{equation*}
By conformal covariance, $\hg$ is scalar-flat; the standard Green function expansion also shows that it is complete and asymptotically flat \cite{LeeParker1987}. Note that $\rho$ is continuous and positive on $M\setminus\{p,q\}$. Moreover, $\rho$ extends continuously to $p$ with $\rho(p)=0$ and tends to infinity at the asymptotically flat end corresponding to $q$. 

\begin{lemma}\label{lem:rho equations}
On $M\setminus\{p,q\}$, we have
\begin{equation}\label{eq:lem2.1}
\hd\rho=3\rho^{-1}|\hn\rho|^2_{\hg},\qquad \hd\rho^2=8|\hn\rho|_{\hg}^2.
\end{equation}
\end{lemma}

\begin{proof}
The conformal covariance of the conformal Laplacian and the scalar-flatness of $\hg$ give
\begin{equation*}
0=L_{\hg}\left(\frac{G_p}{G_q}\right)=-\hd\rho^{-2}.
\end{equation*}
Since 
\begin{equation*}
\hd\rho^{-2}=-2\rho^{-3}\hd\rho+6\rho^{-4}|\hn\rho|_{\hg}^2,
\end{equation*}
the first identity follows. Then,
\begin{equation*}
\hd\rho^2=2\rho\hd\rho+2|\hn\rho|_{\hg}^2=8|\hn\rho|_{\hg}^2,
\end{equation*}
which proves the second assertion.
\end{proof}

The following lemma replaces the Kato inequality used in \cite{LP25}. Its proof uses the conformally Einstein Weitzenb\"ock formula in \cite{Wu2017}.

\begin{lemma}\label{lem:Kato}
Define $ V:=G_q^{-1}|\hw|^2_{\hg}$. Then there exists a universal constant $C>0$ such that
\begin{equation}\label{eq:lem2.2}
\hd V\ge-C|\hw|_{\hg}V.
\end{equation} 
\end{lemma}

\begin{proof}
Let
\begin{equation*}
f=2\log G_q\quad\text{on $\hm$}
\end{equation*}
so that $g=e^{-f}\hg$. We note that $\ric_g=3g$ implies
\begin{equation*}
\ric_{\hg}+\wh{\hess}f+\frac{1}{2}df\otimes df=\lambda\hg,
\end{equation*}
where
\begin{equation}\label{eq:lem2.2-1}
4\lambda=\hd f+\frac{1}{2}|\hn f|_{\hg}^2.
\end{equation}
On the other hand, the scalar curvature transformation for $g=e^{-f}\hg$ implies 
\begin{equation}\label{eq:lem2.2-2}
\hd f-\frac{1}{2}|\hn f|_{\hg}^2=4G_q^{-2}.
\end{equation}
Combining \eqref{eq:lem2.2-1} and \eqref{eq:lem2.2-2}, we obtain
\begin{equation}\label{eq:lem2.2-3}
4\lambda-|\hn f|_{\hg}^2=4G_q^{-2}.
\end{equation}

Let 
\begin{equation*}
\hd_f=\hd-\langle\hn f,\hn\cdot\rangle.
\end{equation*}
The conformally Einstein Weitzenb\"ock formula in \cite[Corollary 1.1]{Wu2017} states that
\begin{equation*}
\hd_f|\hw^{\pm}|^2_{\hg}=2|\hn\hw^{\pm}|^2_{\hg}+\left(4\lambda-|\hn f|^2_{\hg}\right)|\hw^{\pm}|^2_{\hg}-144\det \hw^{\pm}
\end{equation*}
on an oriented coordinate neighborhood. The estimate $|\det\hw^{\pm}|\le C|\hw^{\pm}|^3_{\hg}$, together with \eqref{eq:lem2.2-3}, yields 
\begin{equation}\label{eq:lem2.2-4}
\hd_f|\hw|^2_{\hg}\ge 4G_q^{-2}|\hw|^2_{\hg}-C|\hw|^3_{\hg}.
\end{equation}
Note that the inequality is independent of the chosen local orientation and is global.

A direct computation gives
\begin{align*}
e^{-f/2}\hd_f|\hw|^2_{\hg}&=e^{-f/2}\hd_f(e^{f/2}V)\\&=\hd V+\left(\frac{1}{2}\hd f-\frac{1}{4}|\hn f|^2_{\hg}\right)V\\&=\hd V+2G_q^{-2}V,
\end{align*}
where we used \eqref{eq:lem2.2-2}. Multiplying \eqref{eq:lem2.2-4} by $e^{-f/2}$ then gives
\begin{equation*}
\hd V+2G_q^{-2}V\ge 4G_q^{-2}V-C|\hw|_{\hg}V,
\end{equation*}
which proves \eqref{eq:lem2.2}.
\end{proof}

We assume that, for some $\delta>0$, 
\begin{equation}\label{eq:assumption sec1}
|\hn\rho|_{\hg}\ge\delta\quad\text{on $M\setminus\{p,q\}$.}
\end{equation}
Every positive value of $\rho$ is then regular. We define the function $F:[0,\infty)\rightarrow\R$ by
\begin{align}
F(r)=\int_{\sls}V|\hn\rho|_{\hg}^2\dv.
\end{align}
Each sublevel set $\{\rho\le r\}$ is a compact subset of $\hm$, so $F$ is well-defined.

We use the following notation for asymptotic expansions. For a function $f(r,\theta)$ in geodesic polar coordinates, we denote that $f=O_k(r^\mu)$ if
\begin{equation*}
|\partial_r^j\nabla_\theta^\alpha f(r,\theta)|\le C_{j,\alpha} r^{\mu-j}\quad (j+|\alpha|\le k)
\end{equation*}
uniformly in $\theta$. The Green function expansion \cite{LeeParker1987} gives
\begin{equation}\label{eq:asymptotic of G}
G_p(r,\theta)=r^{-2}(1+O_1(r^2)).
\end{equation}
By symmetry of the Green function, set
\begin{equation*}
\gamma:=G_q(p)=G_p(q)>0.
\end{equation*}

\begin{lemma}\label{lem:graphical expression}
Let $s=d_g(p,\cdot)$ and $t=d_g(q,\cdot)$. There exist two smooth functions $A_p, A_q\in C^\infty(\S^3)$ such that 
\begin{align}
\label{eq:rho near p} \rho(s,\theta)=\gamma^{1/2}s+A_p(\theta)s^2+O_1(s^3)&\quad \text{as } s\rightarrow0,\\
\label{eq:rho near q} \rho(t,\theta)=\gamma^{-1/2}t^{-1}+A_q(\theta)+O_1(t)&\quad\text{as } t\rightarrow0.
\end{align}
In particular, for all sufficiently small $r>0$, the level set $\ls$ near $p$ can be expressed as the polar graph
\begin{equation*}
s=\sigma(r,\theta),\quad \sigma(r,\theta)=\gamma^{-1/2}r+O_1(r^2),\quad \partial_r\sigma=\gamma^{-1/2}+O(r),
\end{equation*}
and for all sufficiently large $r>0$, the same level set near $q$ can be expressed as the polar graph
\begin{equation*}
t=\tau(r,\theta),\quad \tau(r,\theta)=\gamma^{-1/2}r^{-1}+O_1(r^{-2}),\quad \partial_r\tau=-\gamma^{-1/2}r^{-2}+O(r^{-3}).
\end{equation*}
Moreover, $\sls=\{s\le\sigma(r,\theta)\}$ for all sufficiently small $r$, whereas $\{\rho>r\}=\{0<t<\tau(r,\theta)\}$ for all sufficiently large $r$.
\end{lemma}

\begin{proof}
Noting that $G_q$ is positive and smooth near $p$, the Taylor expansion in the normal coordinates, together with \eqref{eq:asymptotic of G}, gives
\begin{equation*}
G_q(s,\theta)=\gamma+B_p(\theta)s+O_2(s^2),\quad G_p^{-1/2}(s,\theta)=s(1+O_1(s^2))
\end{equation*}
for some constant $\gamma>0$ and some smooth function $B_p\in C^\infty(\S^3)$. Expanding $G_q^{1/2}$ and multiplying $G_p^{-1/2}$ proves \eqref{eq:rho near p}. \eqref{eq:rho near q} follows similarly.

\eqref{eq:rho near p} implies $\partial_s\rho=\gamma^{1/2}+O(s)>0$ for all sufficiently small $s$. The implicit function theorem now yields the graph expression. Also, if we substitute the graph $s=\sigma(r,\theta)$ in \eqref{eq:rho near p} and differentiate in $r$, then we get the desired expansions of $\sigma$ and $\partial_r\sigma$. Similarly, equation \eqref{eq:rho near q} gives $\partial_t\rho^{-1}=\gamma^{1/2}+O(t)>0$ for small $t$, and applying the implicit function theorem to $\rho^{-1}$ gives the graph expression $t=\tau(r,\theta)$ and the desired expansions. 
\end{proof}

\begin{lemma}
With the notations of Lemma \ref{lem:graphical expression},
\begin{align}
\label{eq:volume element near p}
V|\hn\rho|_{\hg}^2\dv=\left(\gamma^{-2}|W(p)|^2_gs^3+O(s^4)\right)\,ds\,d\theta&\quad\text{as $s\rightarrow0$},\\\label{eq:volume element near q}
V|\hn\rho|_{\hg}^2\dv=\left(\gamma^{-1}|W(q)|^2_g t^5+O(t^6)\right)\,dt\,d\theta&\quad\text{as $t\rightarrow0$},
\end{align}
where $d\theta$ is the standard volume measure on $\S^3$.
\end{lemma}

\begin{proof}
Since $\hg=G_q^2g$, we have
\begin{equation*}
|\wh{W}|_{\hg}^2=G_q^{-4}|W|_g^2,\quad |\hn \rho|_{\hg}^2=G_q^{-2}|\nabla \rho|_g^2,\quad \dv=G_q^4\,dV_g,
\end{equation*}
and thus
\begin{align}\label{eq:lem2.4-1}
V|\hn\rho|_{\hg}^2\dv=G_q^{-3}|W|^2_g|\nabla\rho|_g^2\,dV_g.
\end{align}

In four-dimensional geodesic polar coordinates centered at either pole $p$ or $q$, we have
\begin{equation}\label{eq:asymptotic of vol}
dV_g=r^{3}(1+O_2(r^2))\,dr\,d\theta.
\end{equation}
Then, at $p$, equations \eqref{eq:rho near p} and \eqref{eq:asymptotic of vol} imply
\begin{equation*}
|\nabla\rho|_g^2=\gamma+O(s),\quad G_q^{-3}=\gamma^{-3}+O(s),\quad |W|^2_g=|W(p)|^2_g+O(s),
\end{equation*}
and substitution in \eqref{eq:lem2.4-1} proves \eqref{eq:volume element near p}. \eqref{eq:volume element near q} follows similarly.
\end{proof}

\begin{proposition}[Asymptotics of $F$ near $0$]\label{prop:asymptotic of F near 0}
As $r\rightarrow0$, we have
\begin{align}
    \label{eq:asymptotic of F near 0}F(r)&=\frac{\omega_3}{4}\gamma^{-4}|W(p)|^2_g\,r^4+O(r^5),\\
    \label{eq:asymptotic of F' near 0}F'(r)&=\omega_3\gamma^{-4}|W(p)|_g^2\,r^3+O(r^4).
\end{align}
In particular, $F(r)=O(r^4)$ and $F'(r)=O(r^3)$ as $r\rightarrow0$.
\end{proposition}

\begin{proof}
For small $r$, Lemma \ref{lem:graphical expression} and \eqref{eq:volume element near p} give
\begin{equation*}
F(r)=\int_{\S^3}\int_0^{\sigma(r,\theta)}\left(\gamma^{-2}|W(p)|^2_g\,s^3+O(s^4)\right)\,ds\,d\theta,
\end{equation*}
where the remainder is uniform in $\theta$. Since $\sigma(r,\theta)=\gamma^{-1/2}r+O(r^2)$, we have
\begin{align*}
F(r)&=\int_{\S^3}\left(\frac{\gamma^{-2}|W(p)|_g^2}{4}\sigma(r,\theta)^4+O(\sigma(r,\theta)^5)\right)\,d\theta\\
&=\frac{\omega_3}{4}\gamma^{-2}|W(p)|^2_g\left(\gamma^{-2}r^4+O(r^5)\right)+O(r^5),
\end{align*}
which is \eqref{eq:asymptotic of F near 0}. 

On the other hand, differentiating $F$ in $r$ yields that
\begin{align*}
F'(r)&=\int_{\S^3}\left(\gamma^{-2}|W(p)|^2_g\sigma^3+O(\sigma^4)\right)\partial_r\sigma \,d\theta\\
&=\omega_3\gamma^{-2}|W(p)|^2_g\left(\gamma^{-3/2}r^3+O(r^4)\right)\left(\gamma^{-1/2}+O(r)\right)+O(r^4)\\
&=\omega_3\gamma^{-4}|W(p)|_g^2r^3+O(r^4),
\end{align*}
which proves \eqref{eq:asymptotic of F' near 0}.
\end{proof}

\begin{proposition}[Asymptotics of $F$ as $r\rightarrow\infty$]\label{prop:asymptotic of F at infinity}
The integral
\begin{equation*}
F_\infty:=\int_{\hm}V|\hn\rho|_{\hg}^2\dv
\end{equation*}
is finite, and $F(r)\rightarrow F_\infty$ as $r\rightarrow\infty$. More precisely,
\begin{equation}
\label{eq:asymptotic of F at infinity}
F_\infty-F(r)=\frac{\omega_3}{6}\gamma^{-4}|W(q)|_g^2\,r^{-6}+O(r^{-7})\quad\text{as } r\rightarrow\infty.
\end{equation}
In particular, $F_\infty-F(r)=O(r^{-6})$ as $r\rightarrow\infty$. 
\end{proposition}

\begin{proof}
The volume element has the form $O(s^3)ds\,d\theta$ near $p$ by \eqref{eq:volume element near p} and $O(t^5)\,dt\,d\theta$ near $q$ by \eqref{eq:volume element near q}. Since it is smooth on each compact subset of $M\setminus\{p,q\}$, it is therefore integrable on $\hm$, and $F_\infty<\infty$. The monotone convergence theorem implies $F(r)\rightarrow F_\infty$ as $r\rightarrow\infty$.

For sufficiently large $r$, Lemma \ref{lem:graphical expression} and \eqref{eq:volume element near q} give
\begin{align*}
F_\infty-F(r)&=\int_{\{\rho>r\}}V|\hn\rho|_{\hg}^2\dv\\
&=\int_{\S^3}\int_0^{\tau(r,\theta)}\left(\gamma^{-1}|W(q)|^2_g\,t^5+O(t^6)\right)\,dt\,d\theta.
\end{align*}
Noting $\tau(r,\theta)=\gamma^{-1/2}r^{-1}+O(r^{-2})$, we have
\begin{align*}
F_\infty-F(r)&=\int_{\S^3}\left(\frac{\gamma^{-1}|W(q)|_g^2}{6}\tau(r,\theta)^6+O(\tau(r,\theta)^7)\right)\,d\theta\\
&=\frac{\omega_3}{6}\gamma^{-1}|W(q)|^2_g\left(\gamma^{-3}r^{-6}+O(r^{-7})\right)+O(r^{-7})\\
&=\frac{\omega_3}{6}\gamma^{-4}|W(q)|^2_g r^{-6}+O(r^{-7}),
\end{align*}
which proves \eqref{eq:asymptotic of F at infinity}.
\end{proof}

We now prove Theorem \ref{thm1} by an ODE argument similar in spirit to \cite{LP25}.

\begin{proof}[Proof of Theorem \ref{thm1}] We recall that
\begin{equation*}
|\hn\rho|_{\hg}\ge\delta\quad\text{on $M\setminus\{p,q\}$}\quad \text{and}\quad \sup_{\hm}\rho^2|\hw|_{\hg}\le c\delta^2
\end{equation*} by assumption. The coarea formula gives 
\begin{equation*}
F'(r)=\int_{\ls}V|\hn\rho|_{\hg}\da.
\end{equation*}
Using the divergence theorem and the first equation in \eqref{eq:lem2.1}, we have
\begin{equation*}
F'(r)=\int_{\sls}\left(\langle\hn V,\hn\rho\rangle+3\rho^{-1}V|\hn\rho|^2_{\hg}\right)\dv.
\end{equation*}
Differentiating and using the coarea formula yield 
\begin{equation*}
F''(r)-\frac{3}{r}F'(r)=\int_{\ls}\partial_\nu V\da=\int_{\sls}\hd V\dv.
\end{equation*}
Next, integration by parts and the second equation in \eqref{eq:lem2.1} give
\begin{align*}
\int_{\sls}\rho^2\hd V\dv&=r^2\int_{\ls}\partial_\nu V\da-\int_{\sls}\langle\hn\rho^2, \hn V\rangle\dv\\
&=r^2\left(F''-\frac{3}{r}F'\right)-2rF'+8F.
\end{align*}
Hence we have
\begin{equation*}
r^2F''(r)-5rF'(r)+8F(r)=\int_{\sls}\rho^2\hd V\dv.
\end{equation*}
Let 
\begin{equation*}
\omega=\sup_{\hm}\rho^2|\hw|_{\hg}.
\end{equation*}
Lemma \ref{lem:Kato}, together with the assumption \eqref{eq:assumption sec1}, yields that 
\begin{align*}
r^2F''-5rF'+8F&\ge-C\int_{\sls}\rho^2|\hw|_{\hg}V\dv\\&\ge -C\omega\int_{\sls}V\dv\\
&\ge-C\delta^{-2}\omega F(r).
\end{align*}
If we let $\Theta=C\delta^{-2}\omega$, then
\begin{equation}\label{eq:ODE}
r^2F''-5rF'+(8+\Theta)F\ge0.
\end{equation}

We now choose $c>0$ so that $0\le \Theta<1$; since $\Theta=C\delta^{-2}\omega$ and $\omega \le c\delta^2$ by hypothesis, it suffices to take $c=1/(2C)$, where $C$ is the universal constant of Lemma \ref{lem:Kato}. If $\Theta=0$, then $\omega=0$ and $\wh{W}\equiv0$, as desired. Thus, we assume $0<\Theta<1$. Let $x\ge y$ be the roots of the homogeneous indicial equation
\begin{equation*}
\lambda^2-6\lambda+(8+\Theta)=0,
\end{equation*}
that is,
\begin{equation*}
x=3+\sqrt{1-\Theta},\qquad y=3-\sqrt{1-\Theta},
\end{equation*}
and $2<y<3<x<4$. If we define
\begin{align*}
h(r)=\frac{F(r)}{r^x},
\end{align*}
then the inequality \eqref{eq:ODE} implies
\begin{equation}\label{eq:ODE2}
\left(\frac{h'}{r^{y-x-1}}\right)'=\left(\frac{rF'-xF}{r^y}\right)'=r^{-y-1}(r^2F''-5rF'+(8+\Theta)F)\ge0.
\end{equation}
By Proposition \ref{prop:asymptotic of F near 0}, we have
\begin{equation*}
\frac{rF'-xF}{r^y}=O(r^{4-y})\rightarrow0\quad\text{as } r\rightarrow0.
\end{equation*}
It then follows from \eqref{eq:ODE2} that $h'(r)\ge0$ for every $r>0$. On the other hand, Proposition \ref{prop:asymptotic of F at infinity} gives 
\begin{align*}
h(r)=\frac{F(r)}{r^x}\rightarrow0\quad\text{as } r\rightarrow\infty. 
\end{align*}
Since $h$ is non-negative and non-decreasing, it must hold that $h\equiv 0$, and hence $F\equiv 0$. By the lower bound \eqref{eq:assumption sec1} and the definition of $V$, we obtain $\hw\equiv0$ on $\hm$.

The conformal invariance of the Weyl tensor implies $W_g\equiv 0$. Since a locally conformally flat Einstein metric has constant sectional curvature, $(M,g)$ has constant sectional curvature 1. By the Killing--Hopf theorem, $(M,g)$ is isometric to a spherical space form $\S^4/\Gamma$ for some finite group $\Gamma\subset \mathrm{O}(5)$ acting freely on $\S^4$.

In dimension four, the spherical space forms are $\S^4$ or $\mathbb{RP}^4$. To exclude the latter we claim that $|\hn\rho|_{\hg}=0$ at some points in the blow-up of $\mathbb{RP}^4$. Denote the covering map by $\pi:\S^4\rightarrow\mathbb{RP}^4$. Fix two distinct points $p,q\in\S^4$ and let $[p]=\pi(p),[q]=\pi(q)\in\RP^4$. Then we have
\begin{equation*}
\pi^\ast G_{[p]}^{\RP^4}=G_{p}^{\S^4}+G_{-p}^{\S^4}.
\end{equation*}
Since we know that $G^{\S^4}_p=(2\sin(d(p,\cdot)/2))^{-2}$ and $d(-p,\cdot)=\pi-d(p,\cdot)$, we have
\begin{equation*}
\pi^\ast G_{[p]}^{\RP^4}=\frac{1}{4\sin^2(d(p,\cdot)/2)}+\frac{1}{4\cos^2(d(p,\cdot)/2)}=\frac{1}{\sin^2d(p,\cdot)}.
\end{equation*}
Thus if we set $u=\langle x,p\rangle, v=\langle x,q\rangle$ on $\S^4\subset\R^5$, then $u^2,v^2$ have the same values at antipodal points, so are well-defined on $\RP^4$. Since $G_{[p]}^{\RP^4}=\frac{1}{1-u^2}$, we have
\begin{equation*}
\rho^2=\frac{G_q}{G_p}=\frac{1-u^2}{1-v^2}.
\end{equation*}

Since $G_{[p]}^{\RP^4}=\frac{1}{1-u^2}$ has its minimum on the cut locus $\mathrm{Cut}(p)=\{u=0\}\cong\RP^3$ where $d(p,\cdot)=\pi/2$, $\nabla G_{[p]}^{\RP^4}=0$ there. For the same reason, $\nabla G_{[q]}^{\RP^4}=0$ on $\mathrm{Cut}(q)=\{v=0\}$. Hence, we have $\nabla G_{[p]}^{\RP^4}=\nabla G_{[q]}^{\RP^4}=0$ on $\mathrm{Cut}(p)\cap\mathrm{Cut}(q)\cong\RP^2$, and $\nabla\rho=0$ and $\hn\rho=0$ there. Therefore, we deduce that $\Gamma$ is trivial and $(M,g)$ is isometric to $\S^4$.  
\end{proof}

\begin{remark}
No relation between $p$ and $q$, other than $p \neq q$, has been used in this section. In particular, the maximality of $b_p$ at $q$ is not assumed. It is only needed in the next section where the hypotheses in Theorem \ref{thm1} are verified under smallness of the integral gap, where we then choose $q$ to be a maximum point of $b_p$ so that the two point become antipodal on the round sphere in the limit.
\end{remark}

\section{Another proof of the four-dimensional gap theorem}\label{section 4}
In this section we prove Theorem \ref{thm4}. We verify that the two hypotheses of Theorem \ref{thm1} are satisfied when the normalized integral gap is small, thereby providing another proof of Corollary \ref{thm3} in dimension four without using the sphere theorem of Honda--Mondello \cite{HondaMondello}.

\begin{proof}[Proof of Theorem \ref{thm4}]
To the contrary, suppose that the conclusion fails for some fixed $0<\delta<2$ and $\zeta>0$. Then there exists a sequence $(M_i,g_i)$ of closed connected Einstein 4-manifolds with $\ric_{g_i}=3g_i$, together with points $p_i,q_i\in M_i$, such that
\begin{equation*}
q_i\in\underset{x\in M_i}{\mathrm{argmax}}\,b_{p_i}(x)\quad\text{and}\quad \E_{p_i}(M_i,g_i)\rightarrow0,
\end{equation*}
but, for every $i$, at least one of the following inequalities fails: 
\begin{equation}\label{eq:thm1.4.1}
\inf_{\hm_i\setminus\{p_i\}}|\hn\rho_{i}|_{\hg_i}\ge\delta,\qquad \sup_{\hm_i}\rho_i^2|\hw_i|_{\hg_i}<\zeta,
\end{equation}
where
\begin{equation*}
\hg_i:=G_{q_i}^2 g_i\quad\text{on $M_i\setminus\{q_i\}$}\quad\text{and}\quad \rho_i:=\left(\frac{G_{q_i}}{G_{p_i}}\right)^{1/2}.
\end{equation*}

By Theorem \ref{thm2} and Bishop--Gromov comparison,
\begin{equation*}
\vol(M_i,g_i)\rightarrow\vol(\S^4).
\end{equation*}
Then Colding's almost-maximal-volume theorem \cite{Colding1996Shape} implies that $(M_i,g_i)$ converges to $\S^4$ in the Gromov--Hausdorff sense. Since $(M_i,g_i)$ are Einstein, up to taking a further subsequence (which we again index by $i$), there exist diffeomorphisms $\Phi_i:\S^4\rightarrow M_i$ such that $\Phi_i^\ast g_i\rightarrow g_{\S^4}$ smoothly, $\Phi_i^{-1}(p_i)\rightarrow p_\infty$, and $\Phi_i^{-1}(q_i)\rightarrow q_\infty$; see for instance \cite{Anderson1990}. The Green functions also converge smoothly on compact sets away from their poles, while their pole expansions are uniform under the convergence; see \cite{LeeParker1987}. If $p_\infty^\ast$ denotes the antipodal point of $p_\infty$, then we have
\begin{equation*}
b_{p_i}(\Phi_i(p_\infty^\ast))\rightarrow\pi.
\end{equation*}
Since $b_{p_i}\le\pi$ and $q_i$ is a maximum point of $b_{p_i}$, we conclude that $b_{p_i}(q_i)\rightarrow\pi$. Since $b_{p_i}$ is 1-Lipschitz and $b_{p_i}(p_i)=0$, we have
\begin{equation*}
b_{p_i}(q_i)\le d_{g_i}(p_i,q_i)\le\pi
\end{equation*}
and $d_{g_i}(p_i,q_i)\rightarrow\pi$. Therefore $q_\infty=p_\infty^\ast$. 

First we consider the Weyl term. The conformal transformation law and the definition of $\rho_i$ yield
\begin{equation*}
\rho_i^2|\wh{W}_i|_{\wh{g}_i}=\frac{|W_i|_{g_i}}{G_{p_i}G_{q_i}}.
\end{equation*}
Note that in dimension four, we have $G_a=(2\sin(b_a/2))^{-2}\ge1/4$ for each $a\in M$. Thus
\begin{equation}
\sup_{\hm_i}\rho_i^2|\hw_i|_{\hg_i}\le16\norm{W_{g_i}}_{L^\infty(M_i,g_i)}\rightarrow0,
\end{equation}
since the limit metric is round and hence has vanishing Weyl tensor. In particular, the second inequality in \eqref{eq:thm1.4.1} holds for all sufficiently large $i$.

It remains to prove a uniform lower bound for $|\hn\rho_i|_{\hg_i}$. On the limiting sphere, denoting $r=d_{\S^4}(p_\infty,\cdot)$, we have
\begin{equation*}
G_{p_\infty}=(2\sin(r/2))^{-2},\quad G_{q_\infty}=(2\cos(r/2))^{-2},\quad \rho_\infty=\tan(r/2)
\end{equation*}
and so
\begin{equation}\label{eq:thm1.4.2}
|\hn\rho_\infty|_{\hg_\infty}=G_{q_\infty}^{-1}|\nabla\rho_\infty|_{g_{\S^4}}=2 \quad\text{on $\S^4\setminus\{p_\infty,q_\infty\}$}.
\end{equation}

Set 
\begin{equation*}
    \gamma_i:=G_{q_i}(p_i)=G_{p_i}(q_i).
\end{equation*}
By smooth convergence of Green functions and noting $q_\infty=p_\infty^\ast$, we have
\begin{equation}\label{eq:thm1.4.3}
\gamma_i\rightarrow G_{q_\infty}(p_\infty)=\frac{1}{4}.
\end{equation}
Let $s_i=d_{g_i}(p_i,\cdot)$. The uniform pole expansions give, near $p_i$,
\begin{equation*}
G_{p_i}=s_i^{-2}(1+O_1(s_i^2)),\qquad G_{q_i}=\gamma_i+O_1(s_i),
\end{equation*}
where the constants in the $O_1$ terms are independent of $i$ for all sufficiently large $i$. Hence
\begin{equation*}
\rho_i=\gamma_i^{1/2}s_i+O_1(s_i^2),\qquad G_{q_i}^{-1}=\gamma_i^{-1}+O(s_i),
\end{equation*}
and 
\begin{equation}\label{eq:thm1.4.4}
|\hn\rho_i|_{\hg_i}=G_{q_i}^{-1}|\nabla\rho_i|_{g_i}=\gamma_i^{-1/2}+O(s_i)
\end{equation}
near $p_i$, uniformly in $i$.

Similarly, if $t_i=d_{g_i}(q_i,\cdot)$, then near $q_i$,
\begin{equation*}
G_{q_i}=t_i^{-2}(1+O_1(t_i^2)),\qquad G_{p_i}=\gamma_i+O_1(t_i),
\end{equation*}
so that
\begin{equation*}
\rho_i=\gamma_i^{-1/2}t_i^{-1}+O_1(1),\qquad G_{q_i}^{-1}=t_i^2(1+O(t_i^2)).
\end{equation*}
It then follows that
\begin{equation}\label{eq:thm1.4.5}
|\hn\rho_i|_{\hg_i}=\gamma_i^{-1/2}+O(t_i)
\end{equation}
near $q_i$, again uniformly in $i$.

By \eqref{eq:thm1.4.3}, \eqref{eq:thm1.4.4} and \eqref{eq:thm1.4.5}, for $r_0>0$ sufficiently small and $i$ sufficiently large, we have
\begin{equation}\label{eq:thm1.4.7}
|\hn\rho_i|_{\hg_i}>\delta\quad\text{on } (B_{r_0}^{g_i}(p_i)\setminus\{p_i\})\cup(B_{r_0}^{g_i}(q_i)\setminus\{q_i\}). 
\end{equation}
On the compact complement of these two balls, the Green functions and their first derivatives converge smoothly to the Green function on $\S^4$. By \eqref{eq:thm1.4.2},
\begin{equation}\label{eq:thm1.4.8}
|\hn\rho_i|_{\hg_i}\rightarrow 2\quad\text{uniformly on } M_i\setminus(B_{r_0}^{g_i}(p_i)\cup B_{r_0}^{g_i}(q_i)).
\end{equation}
Since $\delta<2$, \eqref{eq:thm1.4.7} and \eqref{eq:thm1.4.8} show that the first inequality in \eqref{eq:thm1.4.1} also holds for all sufficiently large $i$, which is a contradiction. This proves Theorem \ref{thm4}. 
\end{proof}

Finally, we remark that taking any $0<\delta<2$ and then choosing $\zeta<c\delta^2$, Theorems \ref{thm1} and \ref{thm4} together yield another proof of Corollary \ref{thm3} in dimension 4.

\bibliographystyle{alpha}
\bibliography{reference2}

@article{Col12,
    AUTHOR = {Colding, Tobias Holck},
     TITLE = {New monotonicity formulas for {R}icci curvature and applications. {I}},
   JOURNAL = {Acta Math.},
    VOLUME = {209},
      YEAR = {2012},
    NUMBER = {2},
     PAGES = {229--263},
}

@misc{Man25,
    AUTHOR = {Manea, Cosmin},
     TITLE = {Sharp gradient estimates and monotonicity in positive {R}icci curvature},
      NOTE = {Preprint, arXiv:2512.05467},
      YEAR = {2025},
}

@article{LP25,
    AUTHOR = {Lee, Sanghoon and Park, Jiewon},
     TITLE = {Rigidity of the gradient estimate for {E}instein manifolds},
   JOURNAL = {J. Geom. Anal.},
    VOLUME = {36},
      YEAR = {2026},
    NUMBER = {10},
     PAGES = {316},
}

@article{Wu2017,
    AUTHOR = {Wu, Peng},
     TITLE = {A {W}eitzenb\"{o}ck formula for canonical metrics on four-manifolds},
   JOURNAL = {Trans. Amer. Math. Soc.},
    VOLUME = {369},
      YEAR = {2017},
    NUMBER = {2},
     PAGES = {1079--1096},
}

@article{LeeParker1987,
    AUTHOR = {Lee, John M. and Parker, Thomas H.},
     TITLE = {The {Y}amabe problem},
   JOURNAL = {Bull. Amer. Math. Soc. (N.S.)},
    VOLUME = {17},
      YEAR = {1987},
    NUMBER = {1},
     PAGES = {37--91},
}

@article{NaberValtorta2017,
    AUTHOR = {Naber, Aaron and Valtorta, Daniele},
     TITLE = {Volume estimates on the critical sets of solutions to elliptic {PDE}s},
   JOURNAL = {Comm. Pure Appl. Math.},
    VOLUME = {70},
      YEAR = {2017},
    NUMBER = {10},
     PAGES = {1835--1897},
}

@article{WeiYe2007,
    AUTHOR = {Wei, Guofang and Ye, Rugang},
     TITLE = {A {N}eumann type maximum principle for the {L}aplace operator on compact {R}iemannian manifolds},
   JOURNAL = {J. Geom. Anal.},
    VOLUME = {19},
      YEAR = {2009},
     PAGES = {719--736},
}

@article{HondaMondello,
    AUTHOR = {Honda, Shouhei and Mondello, Ilaria},
     TITLE = {Sphere theorems for {RCD} and stratified spaces},
   JOURNAL = {Ann. Sc. Norm. Super. Pisa Cl. Sci. (5)},
    VOLUME = {22},
      YEAR = {2021},
    NUMBER = {2},
     PAGES = {903--923},
}

@article{CheegerColding1997,
    AUTHOR = {Cheeger, Jeff and Colding, Tobias H.},
     TITLE = {On the structure of spaces with {R}icci curvature bounded below. {I}},
   JOURNAL = {J. Differential Geom.},
    VOLUME = {46},
      YEAR = {1997},
    NUMBER = {3},
     PAGES = {406--480},
}

@article{Aubry2005,
    AUTHOR = {Aubry, Erwann},
     TITLE = {Pincement sur le spectre et le volume en courbure de {R}icci positive},
   JOURNAL = {Ann. Sci. \'{E}cole Norm. Sup. (4)},
    VOLUME = {38},
      YEAR = {2005},
    NUMBER = {3},
     PAGES = {387--405},
}

@article{Colding1996Shape,
  author  = {Colding, Tobias H.},
  title   = {Shape of manifolds with positive {R}icci curvature},
  journal = {Invent. Math.},
  volume  = {124},
  year    = {1996},
  number  = {1-3},
  pages   = {175--191},
  doi     = {10.1007/s002220050049}
}

@article{Anderson1990,
  author  = {Anderson, Michael T.},
  title   = {Convergence and rigidity of manifolds under {R}icci curvature bounds},
  journal = {Invent. Math.},
  volume  = {102},
  year    = {1990},
  number  = {2},
  pages   = {429--445},
  doi     = {10.1007/BF01233434}
}

@article{HondaMondino,
    author = {Honda, Shouhei and Mondino, Andrea},
    title = {Gap phenomena under curvature restrictions},
    journal = {Indag. Math. (N.S.)},
    volume = {37},
    number = {3},
    year = {2026},
    pages = {744--763},
    doi = {10.1016/j.indag.2025.03.008}
}

@book{Besse,
    author = {Besse, Arthur L.},
    title = {Einstein manifolds},
    series = {Ergebnisse der Mathematik und ihrer Grenzgebiete (3)},
    volume = {10},
    publisher = {Springer-Verlag},
    address = {Berlin},
    year = {1987}
}

@article{HondaPeng,
    author = {Honda, Shouhei and Peng, Yuanlin},
    title = {Sharp gradient estimate, rigidity and almost rigidity of {G}reen functions on non-parabolic {RCD}$(0,{N})$ spaces},
    journal = {Proc. Roy. Soc. Edinburgh Sect. A},
    volume = {155},
    number = {4},
    year = {2025},
    pages = {1267--1320}
}

@misc{GurskyMalchiodi,
  author = {Gursky, Matthew J. and Malchiodi, Andrea},
  title  = {Mass and volume of four-dimensional {E}instein metrics},
  year   = {2025},
  note   = {Preprint, arXiv:2512.07257},
}
\end{document}